\documentclass[11pt,reqno]{amsart}
\usepackage{amsmath}
\usepackage{amsthm}
\usepackage{amssymb,prettyref,appendix,xcolor,graphicx}
\usepackage{comment}
\usepackage{microtype}
\makeatletter
\numberwithin{equation}{section}
\numberwithin{figure}{section}
\theoremstyle{plain}
\newtheorem{thm}{Theorem}[section]
  \newtheorem{lem}[thm]{Lemma}
  \newtheorem{cor}[thm]{Corollary}
  \newtheorem{prop}[thm]{Proposition}
\theoremstyle{definition}
\newtheorem{rem}[thm]{Remark}

\usepackage{mathrsfs,amssymb,bbm,fullpage}
\DeclareFontFamily{U}{mathx}{\hyphenchar\font45}
\DeclareFontShape{U}{mathx}{m}{n}{<-> mathx10}{}
\DeclareSymbolFont{mathx}{U}{mathx}{m}{n}
\DeclareMathAccent{\widebar}{0}{mathx}{"73}
\usepackage[pdfpagemode=UseNone,bookmarksopen=false,colorlinks=true,urlcolor=blue,citecolor=magenta,citebordercolor=blue,linkcolor=blue]{hyperref}

\renewcommand{\liminf}{\varliminf}
\renewcommand{\limsup}{\varlimsup}

\newcommand{\bS}{\mathbb S}

\newcommand{\cC}{\mathcal{C}}

\newcommand{\cN}{\mathcal{N}}
\newcommand{\cE}{\mathcal{E}}

\newcommand{\cG}{\mathcal{G}}

\newcommand{\cP}{\mathcal{P}}

\newcommand{\cR}{\mathcal{R}}

\newcommand{\R}{\mathbb{R}}

\newcommand{\prob}{\mathbb{P}}
\newcommand{\E}{\mathbb{E}}
\newcommand{\eps}{\epsilon}

\newcommand{\indicator}[1]{\mathbbm{1}_{#1}}

\newcommand{\eqdist}{\stackrel{(d)}{=}}
\newcommand{\tensor}{\otimes}

\newcommand{\abs}[1]{\lvert#1\rvert}
\newcommand{\norm}[1]{\lvert\lvert#1\rvert\rvert}

\newcommand{\Var}{\operatorname{Var}}
\newcommand{\gibbs}[1]{\left\langle#1\right\rangle}
\newcommand{\what}[1]{\widehat{#1}}

\DeclareMathOperator{\argmax}{arg\,max}

\newrefformat{prop}{Proposition~\ref{#1}}
\newrefformat{cor}{Corollary~\ref{#1}}
\newrefformat{rem}{Remark~\ref{#1}}
\newrefformat{app}{Appendix~\ref{#1}}

\begin{document}

\title{Statistical thresholds for Tensor PCA}

\author{Aukosh Jagannath}
\author{Patrick Lopatto}
\author{L\'eo Miolane}

\address[Aukosh Jagannath]{Department of Mathematics, Harvard University, Department of Applied Math, University of Waterloo, Department of Statistics and Actuarial Sciences, University of Waterloo}
\email{a.jagannath@uwaterloo.ca}
\address[Patrick Lopatto]{Department of Mathematics, Harvard University}
\email{lopatto@math.harvard.edu}
\address[L\'eo Miolane]{INRIA and D\'epartement d'Informatique de l'\'Ecole Normale Sup\'erieure, Paris}
\email{leo.miolane@gmail.com}

\maketitle
\begin{abstract}
We study the statistical limits of testing and estimation for a rank one deformation of a Gaussian random tensor. 
We compute the sharp thresholds for hypothesis testing and estimation by maximum likelihood and show that they are the same.
Furthermore, we find that {the} maximum likelihood {estimator} achieves the maximal correlation  with the planted vector {among measurable estimators} above the estimation threshold. 
In this setting, the maximum likelihood estimator exhibits a discontinuous BBP-type transition: below the critical threshold the estimator is orthogonal to the planted vector, 
but above the critical threshold, it achieves positive correlation which is uniformly bounded away from zero. 
\end{abstract}

\section{Introduction}

Suppose that we are given an observation, $Y$, which is a $k$-tensor of rank $1$ in dimension $N$ subject to additive Gaussian noise. That is, 
\begin{equation}\label{eq:spiked-tensor-def}
Y=\lambda\sqrt{N}X^{\tensor k}+W,
\end{equation}
where $X\in\mathbb{S}^{N-1}$, the unit sphere in $\R^{N}$, $W$ is an i.i.d.\ Gaussian $k$-tensor with $W_{i_{1}\ldots i_{k}}\sim\cN(0,1)$, and $\lambda\geq0$ is called the signal-to-noise ratio.{\footnote{We note here that none of our results are changed if one symmetrizes $W$, i.e., if we work with the symmetric Gaussian $k$-tensor.}} Throughout this paper, we assume
that $X$ is drawn from an uninformative prior, namely the uniform distribution on $\mathbb S^{N-1}$.
We study the fundamental limits of two natural statistical tasks. The first task is that of hypothesis testing: for what range of $\lambda$ is it statistically possible to distinguish the law of $W$, the null hypothesis, from the law of $Y$, the alternative?
The second task is one of estimation: for what range of $\lambda$ does the maximum likelihood estimator of X, $\widehat{x}^{\rm ML}_{\lambda}(Y)$, achieve asymptotically positive inner product with $X$?

When $k=2$, this amounts to hypothesis testing and estimation for the well-known spiked matrix model.
Here, maximum likelihood estimation corresponds to computing the top eigenvector of $Y$.
This problem was proposed as a natural statistical model of {principal} component analysis \cite{johnstone2001distribution}. 
It is a fundamental result of random matrix theory that there is a critical threshold below 
which the spectral theory of $Y$ and $W$ are asymptotically equivalent, 
but above which the maximum likelihood estimator achieves asymptotically 
positive inner product with $X${---called the {correlation} of the estimator with $X$---}where the correlation increases continuously 
from $0$ to $1$ as $\lambda$ tends to infinity after $N$ 
\cite{edwards1976eigenvalue, baik2005phase,peche2006largest,baik2006eigenvalues,feral2007largest,capitaine2009largest,benaych2011eigenvalues}. 
This transition is called the BBP transition after the authors of \cite{baik2005phase} and has received a tremendous amount of 
attention in the random matrix theory community. Far richer information is also known, such as universality, large deviations, and 
fluctuation theorems. For a small sample of work in this direction, see 
\cite{maida2007,BGM12,bloemendal2013limits,bloemendal2016principal}. More recently, it has been shown that the BBP transition is also the transition for hypothesis testing \cite{montanari2015limitation}. 
See also \cite{deshpande2016asymptotic,barbier2016mutual,perry2016optimality,banks2016information,lelarge2016fundamental,lesieur2017constrained,alaoui2017finite} for analyses of the testing and estimation problem with different prior distributions. 

Our goal in this paper is to understand the case $k\geq3$, which is called the spiked tensor problem. 
This was introduced \cite{richard2014statistical} as a natural generalization of the above to testing 
and estimation problems where the data has more than two indices or requires higher moments, which occurs throughout 
data science \cite{li2010tensor,duchenne2011tensor,anandkumar2014tensor}. 
In this setting, it is known that there is an order 1 lower bound on the threshold for hypothesis testing which is asymptotically tight in $k$ \cite{richard2014statistical,perry2016statistical} and an order 1 upper bound on the threshold for estimation via the maximum likelihood \cite{montanari2015limitation,perry2016statistical}.
On the other hand, if one imposes a more informative, product prior distribution,
 i.e., $X\sim\mu_{0}^{\tensor N}$ for some $\mu_{0}\in\Pr(\R)$, the threshold for minimal mean-square error estimati{on} (MMSE)  has been computed exactly \cite{korada2009exact,lesieur2017statistical,barbier2017stochastic} 
as has the threshold for hypothesis testing under the additional assumption that $\mu_{0}$ is compactly 
supported \cite{perry2016statistical,chen2017phase,chen2018phase}. We note that by a standard 
approximation argument (see \prettyref{prop:IT} below), the results of \cite{lesieur2017statistical,barbier2017stochastic} also imply a sharp threshold $\lambda_{c}$ for which the MMSE 
achieves non-trivial correlation for the uniform prior considered here. 

The authors of \cite{BMMN17} and \cite{BBCR18} began a deep geometric approach to studying this problem by studying the 
geometry of the sub-level sets of the log-likelihood function. {In \cite{BMMN17}, the authors compute the (normalized) logarithm of the expected number of local minima below a certain energy level via the Kac-Rice approach and show that there is a transition at a point $\lambda_s$ such that for $\lambda<\lambda_s$ this is negative
for any strictly positive correlation, and for $\lambda>\lambda_s$ it has a zero with correlation bounded away from zero. {The work} \cite{BBCR18} study the (normalized) logarithm of the (random) number of local minima via a novel (but non rigorous) replica theoretic approach.} In particular, they predict that this problem exhibits a much more dramatic 
transition than the BBP transition. They argued that there are in fact two transitions for the log-likelihood, called $\lambda_{s}$ 
and $\lambda_{c}$.  First, for $\lambda<\lambda_{s}$, all local maxima of the log-likelihood only achieve asymptotically vanishing 
correlation. For $\lambda_s<\lambda<\lambda_{c}$, there is a local maximum of the log-likelihood with non-trivial correlation but 
the maximum likelihood estimator still has vanishing correlation.  Finally, for $\lambda_{c}<\lambda$ the maximum likelihood estimator has strictly positive correlation. In particular, if we let $m(\lambda)$, denote the limiting value of the correlation of the maximum likelihood estimator and $X$, they predict that $m(\lambda)$ has a jump discontinuity at $\lambda_{c}$. 
Finally, they predict that $\lambda_{c}$ should correspond to the hypothesis testing threshold. We verify several of these predictions.

We obtain here the sharp threshold for hypothesis testing and estimation by maximum likelihood and show that they are equal to $\lambda_c$.
Furthermore, we compute the asymptotic correlation between $\widehat{x}^{\rm ML}_\lambda$ and $X$, 
$(\widehat{x}^{\rm ML}_\lambda,X)$, where $(\cdot,\cdot)$ denotes the Euclidean inner product. We find that the maximum likelihood estimator achieves the maximal correlation among
measurable estimators, and that it is discontinuous at $\lambda_c$. This is in contrast to the matrix setting ($k=2)$, where this transition is continuous. 
As a consequence of these results, the threshold $\lambda_c$ is also the threshold for multiple hypothesis testing: the maximum likelihood  is able to distinguish between all of the hypotheses $\lambda>\lambda_c$. 
{Finally, as a consequence of our arguments, we compute the maximum of the log-likelihood for fixed correlation $m$, call it $E_\lambda(m)$, and find that, for $\lambda_s<\lambda<\lambda_c$, $E_\lambda(m)$ has a local maximum at some $m_s>0$.}

These testing and estimation problems have received a tremendous amount of attention recently 
as they are expected to be an extreme example of statistical problems that admit a \emph{statistical-to-algorithmic gap}:
the thresholds for estimation and detection are both order $1$ in $N$; on the other hand,
the thresholds for efficient testing and estimation are expected to diverge polynomially in $N$, $\lambda_{\mathrm{alg}} = O(N^\alpha)$. Indeed, this problem is known to be NP-hard for all $\lambda$ \cite{HiLi13}.
Sharp algorithmic thresholds have been shown for semi-definite and spectral relaxations of the maximum likelihood problem  \cite{hopkins2015tensor,hopkins2016fast,kim2017community} as well as optimization of the likelihood itself via 
Langevin dynamics \cite{BGJ18}. Upper bounds have also been obtained for message passing and power iteration \cite{richard2014statistical}, as well as gradient descent \cite{BGJ18}. 
Our work complements these results by providing sharp statistical thresholds for maximum likelihood estimation and hypothesis testing.

Let us now discuss our main results and methods. 
We begin this paper by computing the sharp threshold for hypothesis testing. 
 There have been two approaches to this in the literature to date. One is by 
a modified second moment method \cite{montanari2015limitation,perry2016statistical}, which yields sharp results in the limit that $k$ tends to infinity after $N$. The other approach, which we take here, is to control the fluctuations of the log-likelihood and yields sharp results for finite $k$. The key idea behind this approach is to prove a correspondence between the statistical threshold for hypothesis testing
and a phase transition, called the ``replica symmetry breaking'' transition, in a corresponding statistical physics problem.
For more on this connection see \prettyref{sec:connection-to-sg} below.

Previous approaches to making this connection precise apply to the bounded i.i.d.\ prior setting 
\cite{korada2009exact,perry2016statistical,chen2017phase,chen2018phase}. There one may apply a techincal, inductive argument of 
Talagrand \cite{talagrand2010meanfield1} related to the ``cavity method'' \cite{mezard1987spin} to control these fluctuations.  This 
approach uses the boundedness and product 
assumption on $\mu_0$ in an essential way, neither of which hold in our setting (though we note here the work 
\cite{panchenko2009cavity} which applies for $\lambda$  sufficiently small). Our main technical contribution in this direction is a simpler, 
large deviations based approach which allows us to obtain the sharp threshold without using the cavity method. 
This argument applies with little modification to the product prior setting as well, though we do not investigate 
this here{. For more on this see \prettyref{rem:LDP}.}

We then turn to computing the threshold for maximum likelihood estimation. We begin by directly computing the 
 almost sure limit of the normalized maximum likelihood, which is an immediate consequence of the results of
 \cite{jagannath2017low,chen2017parisi}. 
Combining this with of the results of \cite{lesieur2017statistical} (and a standard approximation argument), we then obtain a sharp 
 estimate for the correlation between the MLE and $X$ for $\lambda>\lambda_c$ and find that it matches that  of the 
 Bayes-optimal estimator, confirming a prediction from \cite{gillin2000p}. 
 The fact that the MLE has non-trivial correlation down to the information-theoretic threshold $\lambda_c$ is surprising in this setting as it is not expected to be true for all prior distributions. See, e.g., \cite{gillin2001multispin}.

\subsection{Main results}
Let us begin by stating our first result regarding hypothesis testing.
Consider an observation $Y$  of a random tensor. 
Let $P_\lambda^N$ denote the law of \eqref{eq:spiked-tensor-def}.
The null hypothesis is then $P_0^N$ and the alternative $P_\lambda^N$.
Define for $t \in [0,1)$, 
\begin{equation}\label{eq:f-lambda-def}
f_{\lambda}(t) = \lambda^2 t^k +  \log(1-t) + t
\end{equation}
and
let
\begin{equation}\label{eq:lambda_c-def}
	\lambda_c = \sup \Big\{ \lambda \geq 0 \, \Big| \, \sup_{t \in [0,1)} f_{\lambda}(t) \leq 0 \Big\}.
\end{equation}
Our goal is to show that $P_\lambda^N$ and $P_0^N$ are mutually contiguous when $\lambda<\lambda_c$ 
and that for $\lambda>\lambda_c$ there is a sequence of tests $T_N$ which asymptotically distinguish these distributions.
More precisely, we obtain the following stronger result regarding the total variation distance between these hypotheses
which we state in the case $k$ even for simplicity. 
 
\begin{thm}\label{thm:main-thm}
For $k\geq6$ even, 
\[
\lim_{N\to\infty}d_{TV}(P_{0}^{N},P_{\lambda}^{N})
	=\begin{cases} 0 & \text{if} \quad \lambda <\lambda_c\\
	1 & \text{if} \quad \lambda >\lambda_c.
\end{cases}
\]
\end{thm}

The preceding result shows us that the transition for hypothesis testing occurs at $\lambda_c$. 
Let us now turn to the corresponding results regarding maximum likelihood estimation.

It is straightforward to show that maximizing the log-likelihood is equivalent to {maximizing $( Y,x^{\tensor k})$ over the sphere, $x \in \mathbb S^{N-1}$.
The maximum likelihood estimator (MLE) $\what{x}^{\rm ML}_{\lambda}$ is then defined as\protect\footnotemark

\begin{equation}
\what{x}^{\rm ML}_{\lambda} = \argmax_{x\in\mathbb S^{N-1}} ( Y,x^{\tensor k}).
\end{equation}
}
\footnotetext{
As shown in Proposition~\ref{prop:unique_ML}, 
$x \mapsto ( x^{\otimes k}, Y )$ admits almost surely a unique maximizer over $\mathbb{S}^{N-1}$ if $k$ is odd, and two maximizers $x^*$ and $-x^*$ if $k$ is even.
In the case of even $k$, $\what{x}^{\rm ML}_{\lambda}$ is simply picked uniformly at random among $\{-x^*,x^*\}$.
}
Our second result is that the preceding transition is also the transition for which maximum likelihood
estimation yields an estimator which achieves positive correlation with $X$. 

Let $q_*(\lambda)$ be defined by 
\begin{equation}\label{eq:q-star-def}
q_*(\lambda) = \begin{cases}
	0 & \text{if} \quad \lambda <\lambda_c\\
	\argmax_{t\in[0,1)} f_\lambda(t) & \text{if} \quad \lambda>\lambda_c.
\end{cases}
\end{equation}
As shown in Lemma~\ref{lem:gauss_scalar}, the function $f_{\lambda}$ admits a unique positive maximizer on $[0,1)$ when $\lambda > \lambda_c$, so that this is well-defined. 
Let $z_k$ denote the unique zero on $(0,+\infty)$ of
\begin{equation}\label{eq:def_varphi}
	\varphi_k(z) = \frac{1+z}{z^2} \log(1+z) - \frac{1}{z} - \frac{1}{k}.
\end{equation}
Finally, let 
\begin{equation}\label{eq:def_GS}
{\rm GS}_k = \frac{\sqrt k}{\sqrt{1+z_k}} \left(1+\frac{z_k}{k}\right).
\end{equation}

We then have the following.
\begin{thm}\label{thm:max_likelihood}
	Let $\lambda \geq 0$ and $k\geq 3$.
	The following limit holds almost surely
	\begin{equation}\label{eq:lim_max_likelihood}
	\lim_{N \to \infty}
	\frac{1}{\sqrt N} \max_{x \in \mathbb S^{N-1}} \big( x^{\otimes k}, Y \big)
	\ = \
	\begin{cases}
		{\rm GS}_k & \text{if} \ \ \lambda \leq \lambda_c \\
		\displaystyle\sqrt{k}\frac{1 + \lambda^2 q_*(\lambda)^{k-1}}{\sqrt{1 + \lambda^2 k q_*(\lambda)^{k-1}}} & \text{if} \ \ \lambda > \lambda_c.
	\end{cases}
	\end{equation}
Furthermore, we have that {for $\lambda\neq \lambda_c$}
	\begin{equation}\label{eq:correlation}
			\lim_{N\to\infty} \Big| \big( \what{x}_{\lambda}^{\rm ML}, X \big) \Big|	
			 = \sqrt{q_*(\lambda)}.
	\end{equation}
\end{thm}
As a consequence of \prettyref{cor:upper_IT}, the maximum likelihood estimator {achieves maximal correlation}. 
Unlike the case $k=2$, the transition in $q_*(\lambda)$ 
 is not continuous. See Figure~\ref{fig:1}.
  \begin{figure}[h!]
	\centering
	\includegraphics[width=.5\textwidth]{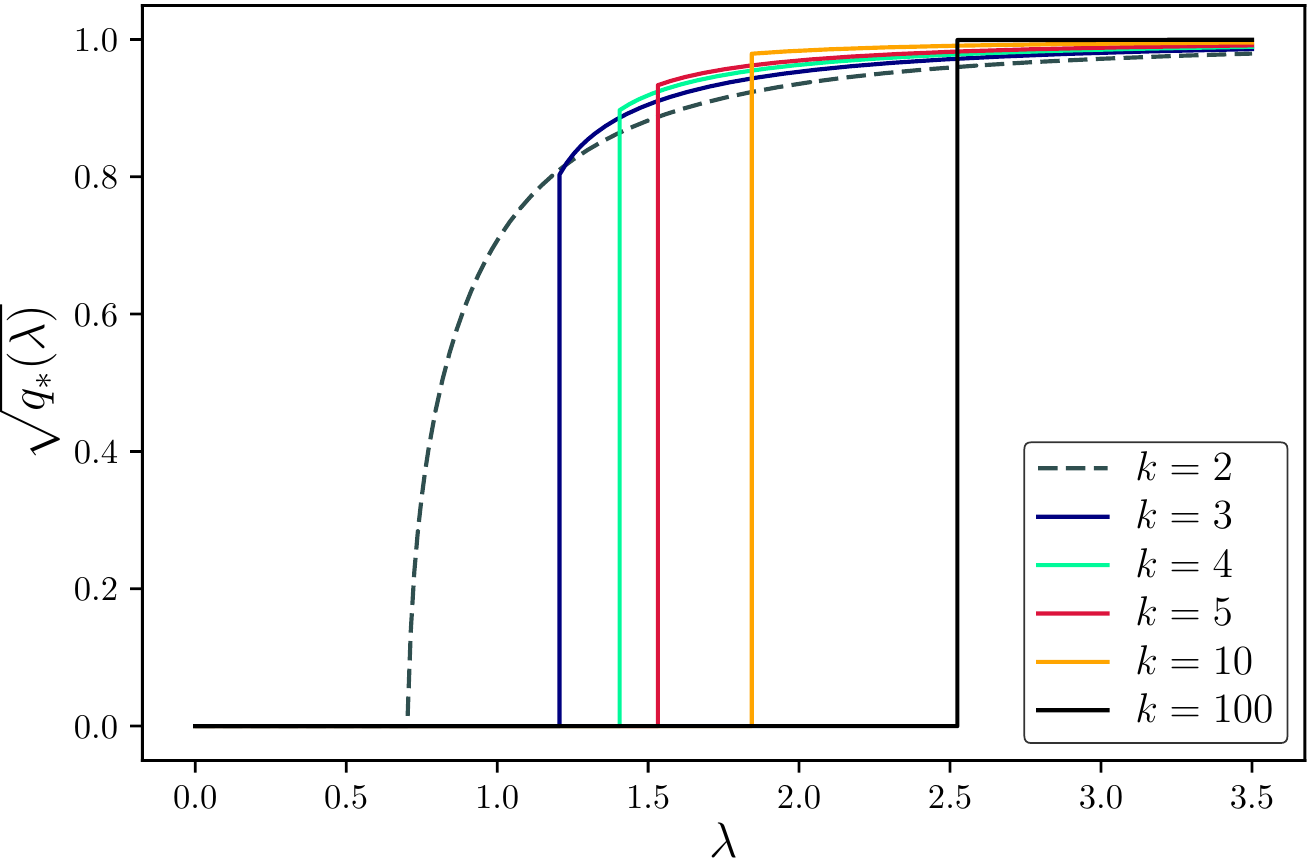}
	\caption{ Asymptotic correlation $\lim\limits_{N \to \infty} \frac{1}{N} | \langle \what{x}^{\rm ML}_{\lambda}, X \rangle | = \sqrt{q_*(\lambda)}$ as a function of the signal-to-noise ratio $\lambda$, for different values of $k$.
	}
	\label{fig:1}
\end{figure}

{Let us pause for a moment to comment on the importance of the assumption that the prior is 
a uniform spherical prior, as opposed to, e.g., a bounded i.i.d.\ prior as in earlier works. 
\begin{rem}\label{rem:LDP}
The argument for the main hypothesis testing result \prettyref{thm:main-thm} does not use in an essential way the spherical nature 
of the prior. The key idea is to show that a certain rate function has a locally quadratic lower bound near its zero, implying a CLT-type upper bound on fluctuations of the corresponding random variable. This estimate extends to rather broad families of problems 
which satisfy the ``positive replicon eigenvalue'' condition. See \cite{arous2018spectral} for more on this condition and how it implies such bounds. However, the sphericity assumption is used in an essential way in the estimation section. The point here is that the variational problem that arises in the spherical setting has a very explicit form. As a result, one can obtain exact expressions for the transition as well as the optimizers. Using these exact expressions, we can verify that the maximum likelihood estimator is information-theoretically optimal, which plays a vital role in our argument. One can obtain a corresponding variational representation for other priors, but the variational problem is in general substantially less transparent. 
\end{rem}}

\subsection*{Regarding the second threshold}
While the regime $\lambda_s<\lambda<\lambda_c$ and the expected transition at $\lambda_s$ {are} not relevant
for testing and estimation, {it has} a natural interpretation from the perspective of the landscape of the maximum likelihood.
In \cite{BMMN17,BBCR18}, this is explained in terms of the complexity. There is also an explanation in terms of the optimization of the maximum likelihood. We end this section with a brief discussion of this phase. 
Let $\lambda_s$ be given by 
 \begin{equation}\label{eq:lambda_s-def}
 \lambda_s = \sqrt{\frac{(k-1)^{k-1}}{k (k-2)^{k-2}}}.
 \end{equation}
Consider the constrained maximum likelihood,
\begin{equation}\label{eq:def_E}
	E_{\lambda}(m) = \lim_{N \to \infty} \frac{1}{N} \max_{x \in \mathbb{S}^{N-1}, \, (X,x) = m} \Big\{ \lambda N (X,x)^k + \sqrt N (W,x^{\tensor k}) \Big\}.
\end{equation}
This limit exists and is given by an explicit variational problem (see \eqref{eq:E-rs} below).  
For $\lambda>\lambda_s$, let ${\sqrt{q_s(\lambda)}}$ be the (unique) positive, strict local maximum of $f_\lambda$. By \prettyref{lem:gauss_scalar}, this is well-defined and satisfies $q_s(\lambda)=q_*(\lambda)$ for $\lambda>\lambda_c$.
In \cite{BBCR18}, it is argued by the replica method that $E_\lambda(m)$ has a local maximum at $\sqrt {q_s(\lambda)}$ for all
$\lambda>\lambda_s$. Establishing this rigorously is a key step in our proof of \prettyref{thm:max_likelihood}. In particular,
we prove the following.

\begin{prop}
For $\lambda>\lambda_s$, the function $E_\lambda$ has a strict local maximum at ${\sqrt{q_s(\lambda)}}$.
{
	It is a global maximum of $E_{\lambda}$ if and only if $\lambda \geq \lambda_c$. For $\lambda \leq \lambda_c$ the global minimum of $E_{\lambda}$ is achieved at $m=0$.
}
\end{prop}
{\noindent The proof of this result immediate{ly} follows by combining \prettyref{lem:up_qs} {with Lemmas \ref{lem:gauss_scalar}--\ref{lem:x_k}} below.}

It is easy to verify (by direct differentiation) that the map $\lambda \mapsto E_\lambda({\sqrt{q_s(\lambda)}})$ is strictly increasing on $(\lambda_s, +\infty)$. 
We have also that $E_{\lambda_c}(\sqrt{q_s(\lambda_c)}) = {\rm GS}_k$ by \prettyref{lem:up_qs} and \prettyref{lem:x_k}, so we get that for $\lambda_s < \lambda < \lambda_c$ the strict local maximum at $\sqrt{q_s(\lambda)}$ has $E_{\lambda}(\sqrt{q_s(\lambda)})$ strictly less than the maximum likelihood.
In fact, \eqref{eq:E-rs} can be solved numerically, as it can be shown 
that one may reduce this variational problem, in the setting we consider here, to a two-parameter family of problems in three real variables. This is discussed in \prettyref{rem:rigorous} below. In particular, see Figure \ref{fig:2} for an illustration of these two transitions in the case $k=4$.
  \begin{figure}[h!]
	\centering
	\includegraphics[width=.5\textwidth]{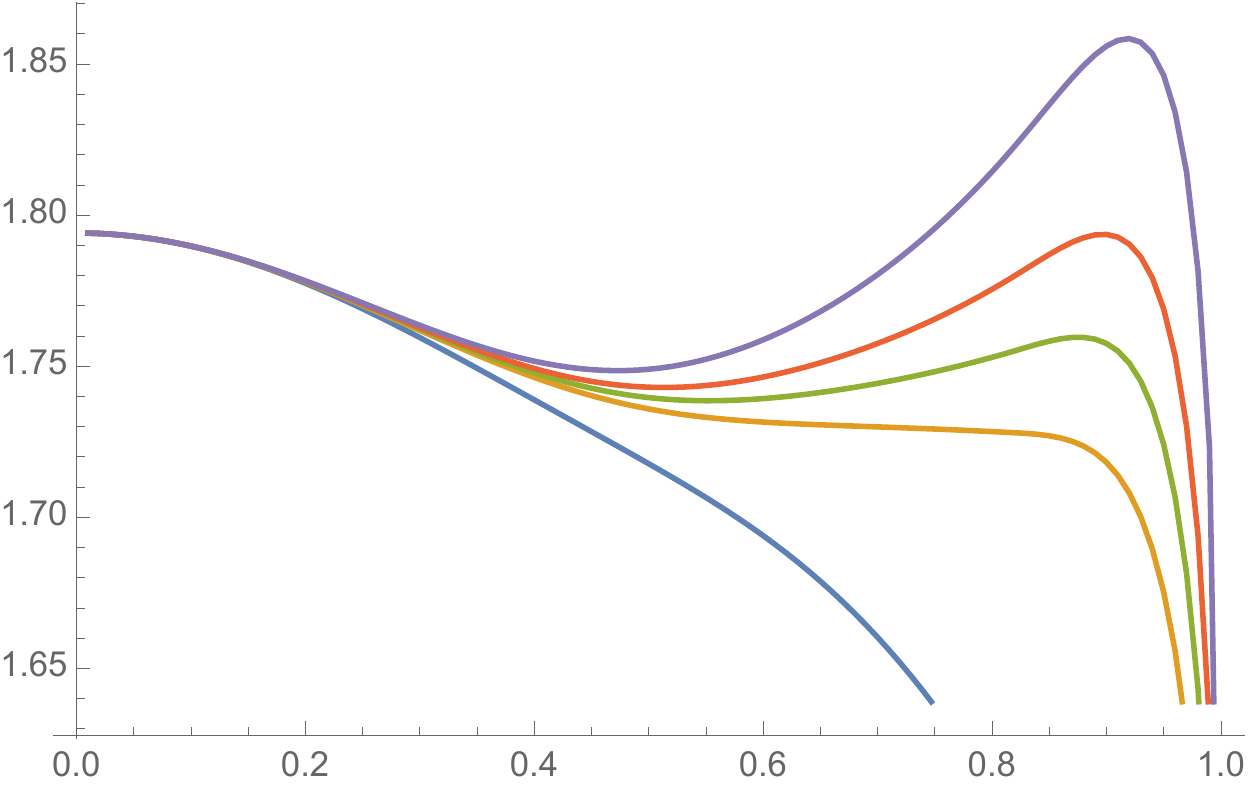}
	\caption{ Asymptotic constrained maximum likelihood $E_\lambda(m)$ for $p=4$ with $\lambda= 1, 1.299,1.35,1.405,1.5$. Here $\lambda_s\approx 1.299$ and $\lambda_c\approx 1.405$. For $\lambda<\lambda_s$, the function is (numerically) seen to be monotone. A secondary maximum occurs at the transition $\lambda=\lambda_s$. This local maximum is bounded away from $m=0$. Finally, at the information theoretic threshold $\lambda_c$, the maximum likelihood is now maximized at this second point.}
	\label{fig:2}
\end{figure}

Let us now compare this with the complexity based approach in \cite{BMMN17,BBCR18}. In \cite{BMMN17}, the authors computed
the expected number of local maxima of a fixed likelihood and correlation (called the annealed 0-complexity) on the exponential in $N$ scale. They showed that for $\lambda<\lambda_s$ this expectation scales like $O(e^{-cN})$ for correlations $q>0$, and $O(\exp(cN))$ for a suitable range of likelihoods when $q=0$, whereas for $\lambda_s<\lambda$, the show that at $q_s$, the logarithm of this quantity is $o(N)$.
Furthermore, they find that once $\lambda>\lambda_c$, this exponent is maximized at a value of likelihood that is larger than the value for correlation $q=0$. While this argument does not show that this behavior is typical, it was argued in \cite{BBCR18},
 via a novel (but non-rigorous) replica method, that the same result holds for the log-number of local maxima (called the quenched $0$-complexity) with high probability. In contrast, in this paper we bypass the analysis of critical points and instead obtain a similar picture by directly computing the almost sure limit of the constrained maximum likelihood, .

 \subsection*{Acknolwedgements}
 A.J. would like to thank G. Ben Arous for encouraging the preparation of this paper. 
 A.J. and L.M. would like to thank the organizers of the BIRS workshop ``Spin Glasses and Related Topics"
 where part of this research was conducted.
This work was conducted while A.J. was supported by NSF OISE-1604232 
and {P.L. was partially supported by the NSF Graduate Research Fellowship Program under grant DGE-1144152.}

\section{Proof of \prettyref{thm:main-thm} and connection to spin glasses}\label{sec:connection-to-sg}
In this section, we prove \prettyref{thm:main-thm}. In particular, we connect
the phase transition for the hypothesis testing problem to a
phase transition in a class of models from statistical physics, which is proved in the remaining sections.

Let us begin by explaining this connection. First note that the null hypothesis is a centered Gaussian distribution on the space of $k$-tensors in $\R^N$,
whereas the alternative corresponds to one with a random mean $\lambda \sqrt N X^{\tensor k}$. 
Thus by Gaussian change of density, the likelihood ratio, $dP_\lambda/dP_0$, satisfies
\[
L(Y) = \frac{dP_\lambda}{dP_0}(Y)  = \int\exp\left( \lambda \sqrt N( Y,x^{\tensor k}) - N\frac{\lambda^2}{2}\right) dx,
\]
where $dx$ denotes the uniform measure on $\mathbb S^{N-1}$.
Observe that the total variation distance satisfies
\begin{equation}\label{eq:total-identity}
d_{TV}(P_\lambda,P_0) =\E_{P_0}\left((1-L(Y))\indicator{\{L(Y)\leq 1\}}\right) = \int_0^1 P_0(L(Y) \leq s)ds.
\end{equation}
We will show that this probability tends to zero almost everywhere when $\lambda < \lambda_c$. 

Let us now make the following change of notation, motivated by statistical physics.
 { For $x\in \mathbb S^{N-1}$ and ${\lambda}\geq0$, define
\begin{equation}\label{eq:ham-def}
H(x)=\sqrt{N}(W,x^{\tensor k}), \quad Z({\lambda})=\int\exp(-{\lambda} H(x))\, dx.
\end{equation}}
We view $H$ as a function on $\mathbb S^{N-1}$, which is called the \emph{Hamiltonian}
of the \emph{spherical ${k}$-spin glass model} in the statistical physics literature \cite{crisanti1992sphericalp}.
The log-likelihood ratio
 then has an interpretation in terms of what is called a ``free energy"  in the statistical
 physics literature. More precisely, define the \emph{free energy at temperature ${\lambda}$}
 for the spherical ${k}$-spin model by
 \begin{equation}
 F_N({\lambda}) =\frac{1}{N}\log Z_N({\lambda})
 \end{equation}
  and observe that  under the null hypothesis, 
   \begin{equation} 
 \frac{1}{N}\log L(Y) = F_N(\lambda) - \frac{\lambda^2}{2}.
 \end{equation}
 The key conceptual step in our proof is to connect the phase transition for hypothesis testing
 to what is called the ``replica symmetry breaking" transition in statistical physics. 
 While it is not within the scope of this paper to provide a complete description of this 
 transition, we note that one expects this transition to be reflected in the limiting properties 
 of $F_N$: if $\lambda$ is small $F_N$ should fluctuate around $\lambda^2/2$,
 but for large $\lambda$ it should be much smaller than $\lambda^2/2$. A sharp transition 
 is expected to occur at $\lambda_c$. For an in-depth discussion of replica symmetry breaking transitions
 see \cite{mezard1987spin}. In the remainder of this section
 we reduce the proof of our main result to the proof that the phase transition
 for the fluctuations of $F_N$ does in fact occur at $\lambda_c$.  We then prove this phase transition exists in the next two sections.

Let us turn to this reduction. By \eqref{eq:total-identity} and the equivalence noted above,
\begin{equation}\label{eq:tv-to-fe}
	d_{TV}(P_{0},P_{\lambda})=\int_{0}^{1}P_0\left(F_N(\lambda)-\frac{\lambda^{2}}{2}<\frac{\log(x)}{N}\right)dx.
\end{equation}
We have the following theorem of Talagrand, which we state in a weak form
for the sake of exposition. {Here and in the following, unless otherwise specified $\mathbb P$ and $\mathbb E$ will always denote
integration with respect to the law of the Gaussian random tensor $W$.}

\begin{thm}[Talagrand \cite{TalSphPF06}]\label{thm:parisi2}
For every ${\lambda}>0$, 
$\E F_N({\lambda})$ is a convergent sequence. Furthermore, 
\[
	\lim_{N\to\infty} \E F_{N} \leq \frac{{\lambda}^{2}}{2},
\]
with equality if and only if ${\lambda}\leq\lambda_c$.
\end{thm}

With this in hand, it suffices to show the following. 
\begin{thm}\label{thm:FE-decay}
For  $k\geq 4$ even, $\eps>0$ and ${\lambda}<\lambda_{c}$,  there is a {constant} $C>0$ such that for every $N\geq 1$ and $x>0$,
\[
\prob\left( \left| F_N({\lambda})-\frac{{\lambda}^{2}}{2} \right| > \frac{x}{N}\right)\leq C\frac{1}{x^2N^{{\frac{k-4}{4}} -\eps}}.
\]
\end{thm}
\begin{proof}
The proof of this theorem will constitute the next two sections.
Let us begin by making the following elementary observations, which will reduce the 
theorem to certain fluctuation theorems. To this end, observe that by 
 Chebyshev's inequality,
 \begin{equation}\label{eq:chebyshev}
\prob\left(\left|F_N({\lambda})-\frac{{\lambda}^{2}}{2}\right|>\eta/N\right)\leq\frac{N^{2}}{\eta^{2}}\left(\Var(F_{N})+\left(\E F_{N}-\frac{{\lambda}^{2}}{2}\right)^{2}\right).
\end{equation}
The key point in the following will be to quantify the rate of convergence in Talagrand's theorem
when $\lambda<\lambda_c$. This rate of convergence will also
allow us to control the variance of $F_{N}$. 
More precisely, in the subsequent sections we will prove  the following two theorems.
\begin{thm}\label{thm:conv-means}
{Fix $k\geq 4 $ even} and ${\lambda}<\lambda_{c}$. For any $\epsilon>0$ there is a $C({\lambda},\epsilon)>0$
such that for $N\geq 1$
\[
\left| \E F_{N}-\frac{{\lambda}^{2}}{2}\right|\leq CN^{-{k/4}+\epsilon}.
\]
\end{thm}

\begin{thm}\label{thm:variance-bound}
	{Fix $k\geq 4$ even.} For ${\lambda}<\lambda_c$ and $\eps>0$, for $N\geq 1$  
\[
\Var(F_N)\leq\frac{C({\lambda})}{N^{{k/4}+1-\eps}}.
\]
\end{thm}
\noindent The desired result then follows upon combining \prettyref{thm:conv-means} and \prettyref{thm:variance-bound} with \eqref{eq:chebyshev}.
\end{proof}

We can now prove the main theorem.
\begin{proof}[\textbf{\emph{Proof of \prettyref{thm:main-thm}}}]
Suppose first that $\lambda <\lambda_c$. Then by \prettyref{thm:FE-decay}
combined with \eqref{eq:tv-to-fe}, the total variation distance vanishes for $k\geq 6$ even. 

Suppose now that $\lambda> \lambda_c$. Note
\[
\prob( F_N \leq \E F_N -\epsilon) \leq \exp (-N\epsilon^2/2)
\]
by Gaussian concentration (see, e.g., \cite[Theorem~5.6]{boucheron2013concentration}). By Jensen's inequality $\E F_N(\lambda) \leq \lambda^2/2$, and by Talagrand's 
theorem, we know that $$\E F_N\to F\neq {\lambda}^2/2,$$ so for some $\eps>0$,
 {
	 $$\lim_{N\to\infty}\mathbb{P}(F_N(\lambda)-\frac{\lambda^2}{2}<-\epsilon)=1.$$ 
 }
The desired result then follows by using this to lower bound the right side of \eqref{eq:tv-to-fe}.
\end{proof}

\section{Rate of convergence of the mean and Decay of variance.}
In this section, we prove \prettyref{thm:conv-means}.
In the following we will make frequent use of the measure 
\[
\pi_{{\lambda}}(dx)\propto\exp(-{\lambda} H(x))\, dx
\]
where $H$ is as in \eqref{eq:ham-def} and $0<{\lambda}$.  We call this the \emph{Gibbs measure}, which we normalize to be a probability
measure. Observe that this normalization constant is given exactly
by $\log Z({\lambda})$. Here and in the remaining sections, we will let $\left\langle \cdot\right\rangle $ denote
expectation with respect to the (random) measure $\pi_{{\lambda}}$. We
will suppress the dependence on ${\lambda}$ whenever it is unambiguous
as it will always be fixed.
Throughout this section, $\lambda$ will always be fixed and less than $\lambda_{c}$.
It will also be useful to define the quantity
\[
	F_{2,N}(u,\eta;{{\lambda}})=\frac{1}{N}\log\int\int_{\abs{\abs{(x,y)}-u}<\eta}\exp(-{\lambda} H(x)-{\lambda} H(y))\, dx\, dy,
\]
where $(x,y)$ denotes the Euclidean inner product. Evidently this is related to the large deviations rate function for the event
$(x,y) \approx u$. To simplify notation, for $X^1,X^2\sim \pi_{{\lambda}}$, we let
\[
R_{12} = (X^1,X^2).
\]

The starting point for our analysis is the estimate of the rate of
convergence of $\E F_N$ to ${\lambda}^2/2$.

\begin{proof}[\bf Proof of Theorem~\ref{thm:conv-means}]
In the following, let 
\[
\psi({\lambda})=\frac{{\lambda}^{2}}{2}-\E F_{N}.
\]
By Jensen's inequality, $\psi({\lambda})\geq0$. 

Let us now turn to an upper bound. Recall $H$ from \eqref{eq:ham-def}.
Observe that $H$ is centered and has covariance 
\[
\E H(x^1)H(x^2) = N(x^1,x^2)^k.
\]
It then follows that
\[
\frac{d}{d{\lambda}} \psi({\lambda}) = {\lambda} -\frac{1}{N} \E \gibbs{H} = {{\lambda}}\E\gibbs{R_{12}^k},
\]
where the first equality is by definition of the Gibbs measure and 
the second follows by Gaussian integration by parts for Gibbs expectations, \eqref{eq:GGIBP}.
We now claim that 
\begin{equation}\label{eq:gronwall-step}
\frac{d}{d{\lambda}}\psi({\lambda})={{\lambda}}\E\left\langle R_{12}^{k}\right\rangle \leq C\psi({\lambda})+\frac{C}{N^{{k/4-k\delta/2}}}.
\end{equation}
{for some constant $C>0$} and $\delta>0$ sufficiently small. With this claim in hand, we may apply Gronwall's inequality and the lower bound from above to obtain
\begin{align*}
0 & \leq\psi({\lambda})\leq\left(\psi(0)+\frac{1}{N^{k/2- k\delta}}\right)\exp\left(C{\lambda}\right)=\frac{C({\lambda})}{N^{{k/2-k\delta/2}}},
\end{align*}
as desired. Let us now turn to the proof of this claim.

Observe that the maps $W\mapsto F_{2,N}(u,\frac{1}{N};{\lambda})$ and
$W\mapsto F_{N}({\lambda})$ are uniformly ${{\lambda}}/\sqrt{N}$-Lipschitz, so
that Gaussian concentration of measure {(see for instance \cite{boucheron2013concentration}, Theorem~5.6)} implies that there
are constants $C,c {>0}$ {depending only on $\lambda$ and $k$} such that for any $\delta \in (0,1/2)$, with probability
at least $1-C\exp(-cN^{2\delta}),$
\begin{align*}
F_{2,N}(u,\frac{1}{N};{\lambda})-\E F_{2,N}(u,\frac{1}{N};{\lambda})  \leq \frac{1}{N^{1/2-\delta}},\qquad
F_{N}({\lambda})-\E F_{N}({\lambda})  \leq\frac{1}{N^{1/2-\delta}}.
\end{align*}
Thus, on this event,
call it $A(u,\delta)$,
\begin{equation}\label{eq:ldp-vs-conv}
\begin{aligned}
\frac{1}{N}\log\pi^{\tensor2}\left(\abs{R_{12}}\in(u-\frac{1}{N},u+\frac{1}{N})\right) & =  F_{2,N}(u,\frac{1}{N},{\lambda})-2F_{N}({\lambda}).\\
																					   & \leq2\left(\frac{{\lambda}^{2}}{2}-\E F_{N}({\lambda})\right)+\left(\E F_{2,N}(u,\frac{1}{N},{{\lambda}})-{\lambda}^{2}\right)+\frac{C}{N^{1/2-\delta}}.
\end{aligned}
\end{equation}
As we shall show in Corollary~\ref{cor:parisi}, for every $\delta>0$ there is some $c>0$ such
that for $N\geq 1$ and for all $N^{-1/2+\delta}\leq u\leq1${;}
\[
\E F_{2,N}(u,\frac{1}{N},{{\lambda}})-{\lambda}^{2}\leq-cu^{2}.
\]
Let 
\[
v=\frac{1}{c}\left(2\psi({\lambda})+\frac{C}{N^{1/2-\delta}}+\frac{1}{\sqrt N}\right).
\]
Then for $u^2\geq v$, on  $A(u,\delta)$,
\begin{equation}
\pi^{\tensor2}(|R_{12}|\in(u-\frac{1}{N},u+\frac{1}{N}))\leq\exp(-\sqrt{N}).\label{eq:decay-sqrt-n-ovlp}
\end{equation}
Consequently, if we take $\{u_i\}_{i=1}^L$ to be the centers of a partition of the interval $[v,1]$ into intervals of size {$2/N$ }, then if we let $A(\delta)=\cap_i A(u_i,\delta)$,
\begin{align*}
	\E\left\langle R_{12}^k\right\rangle  & \leq\E\left(\left\langle R_{12}^k\indicator{R_{12}^{2}>v}\right\rangle \indicator{A(\delta)}\right)+\E\left(\left\langle R_{12}^k\indicator{R_{12}^2\leq v}\right\rangle \indicator{A(\delta)}\right)+Ne^{-cN^{{2}\delta}}.\\
										  & \leq N\exp(-\sqrt{N})+v^{{k/2}}+Ne^{-cN^{{2}\delta}},
\end{align*}
where  we use that $L\leq N$.
From this it follows, by the inequality $(x+y)^k \leq 2^{k-1}(x^k+y^k)$, that
\begin{equation}
\E\left\langle R_{12}^k\right\rangle \leq C\left(\psi({\lambda})^k+\frac{1}{N^{{k/4-k\delta/2}}}\right)\label{eq:fe-conv-ovlp-gronwall}
\end{equation}
for some $C>0$, $\delta$ small enough and $N\geq 1$. The claim \eqref{eq:gronwall-step} then follows since {$\psi(\lambda) \leq \lambda_c^2$ for all $\lambda \leq \lambda_c$}.
For this last claim, observe that $\E F_N(\lambda)$ is convex in $\lambda$ with $\E F_N(0)= 0$ and right derivative $\frac{d}{d\lambda}\E F_N(0^+)=0$, so that $\E F_N\geq0$. 
As a result,  $\psi(\lambda)\leq \lambda^2\leq \lambda_c^2$.
\end{proof}
Notice that by the above argument, we also have the following.
\begin{cor}\label{cor:overlap}
For {any $k\geq 4$ even and} ${\lambda}<\lambda_{c}$ and $\eta>0$, there is a $C({\lambda},\eta)>0$ such that for $N$ sufficiently large,
\[
	\E\left\langle \abs{R_{12}}^k\right\rangle \leq\frac{C({\lambda}{,\eta})}{N^{{k/4 -k\eta/2}}}.
\]
\end{cor}
\begin{proof}
	By \prettyref{thm:conv-means}, we have, {for $N$ large enough},
\[
\psi({\lambda}) \leq \frac{C({\lambda}{,\eta})}{N^{{k/4-k\eta/2}}}.
\]
Combining this with  \eqref{eq:gronwall-step} yields the desired {inequality}.
\end{proof}
We are now in a position to prove the variance decay. 

\begin{proof}[\bf Proof of Theorem~\ref{thm:variance-bound}]
	By the Gaussian Poincar\'e inequality {(see for instance \cite[Theorem 3.20]{boucheron2013concentration})}, 
\begin{align*}
\Var(F_{N}({\lambda})) & \leq\frac{1}{N^{2}}\E\sum_{1 \leq i_{1},\ldots,i_{k} \leq N}\left(\partial_{W_{i_{1}\ldots i_{k}}}\log Z({\lambda})\right)^{2}
  =\frac{{\lambda}^{2}}{N}\E\sum_{1 \leq i_{1},\ldots,i_{k} \leq N}\left\langle x_{i_{1}}\cdots x_{i_{k}}\right\rangle ^{2}
  =\frac{{\lambda}^{2}}{N}\E\left\langle R_{12}^k\right\rangle 
\end{align*}
The result then follows by combining this with Corollary~\ref{cor:overlap}. 
\end{proof}

\section{The Parisi functional and large deviations}

The main technical tool we need is a bound on the following expected
value, which is related to large deviations of $\abs{R_{12}}$ from its mean:
\[
\E\frac{1}{N}\log\pi^{\tensor2}(\abs{R_{12}}\in(u-\eta,u+\eta))=\E F_{2,N}(u,\eta)-2\E F_N.
\]
We relate the quantities $\E F_{2,N}(u,\eta)$ and
$\E F_{N}$ to explicit \emph{Parisi-type} formulas. In the following, let $\xi(t)=\lambda^{2}t^k$ and $\theta(t) = t \xi'(t) - \xi(t)$. For $u, \Lambda \in [ 0 ,1]$ and $m\in [1,2]$, define
\begin{align}
\mathcal P (u, m, \Lambda)  &= \xi(1) + (1-m)\theta(u) - \Lambda u  + \frac{1}{m} \log \frac{1 + \xi'(u) - \Lambda}{1 + (1-m) \xi'(u) - \Lambda} \\& - \frac{1}{2}\left(  \log( 1 + \xi'(u) + \Lambda) + \log ( 1 + \xi'(u) -\Lambda)  \right).
\end{align}
Then we have the following from \cite{TalSphPF06}. See also   \cite{PanchTal07,ko2018free} for alternative presentations. 
\begin{thm}
	For {$k\geq 2$ even}, there exists a constant $C( k)>0$ such that
	for every $N\geq1$, $\eta>0$, $m \in [1,2]$, $\Lambda \in [0,1]$, and $0<u<1$, we have 
\begin{align}
\E F_{2,N}(u,\eta;\lambda) & \leq \mathcal P (u,m, \Lambda )+\mathcal{R}\label{eq:bound-1}
\end{align}
where $ | \mathcal{R} | \le C\eta+\frac{C\log N}{N}$. 
\end{thm}
\begin{proof}
We first observe that by symmetry of $H(x)$, it suffices to prove the same estimate for 
\[
\tilde F_{2,N}(u,\eta) = \frac{1}{N}\log\pi^{\tensor 2}( R_{12} \in (u-\eta,u+\eta)).
\]
We apply  \cite[Eq. 2.22]{PanchTal07} 
with the choice of parameters
\begin{align*}
Q^{0}  =Q^{1}=Q^2=0,\quad
Q^{3}  =\left(\begin{array}{cc}
u & u\\
u & u
\end{array}\right),\quad
Q^{4}  =\left(\begin{array}{cc}
1 & u\\
u & 1
\end{array}\right),
\end{align*}
\begin{equation}
\mathbf{m}=(0,1/2,m/2,1), \quad A_3 = \left(\begin{array}{cc}
1 + \xi'(u) & -\Lambda\\
-\Lambda  & 1 + \xi'(u)
\end{array}\right),
\end{equation}
 to obtain
 \[
 \E {\tilde{F}}_{2,N}(u,\eta)  \leq \cP(u,m,\Lambda) + \cR,
 \]
where the error term $\mathcal{R}$ in \cite[Eq. 2.22]{PanchTal07} 
satisfies 
\[
\mathcal R = {C \eta}  -  2 \left ( \frac{1 }{N}\log \mathbb P \left(\sum_{i=1}^N X_i^2 \geq N\right) + \frac{b - 1 - \log b}{2} \right),
\]
where $b=(1 + \xi'(u))$, the $X_i$ are i.i.d.\ gaussian random variables with variance $1/b$, as given in  \mbox{\cite[Eq. 2.14]{PanchTal07}}, and $C>0$ is universal. Using the elementary bound of Lemma~\ref{lem:chi_lb}, 
it follows that \mbox{$\mathcal |R| \leq C \left(\eta + \frac{\log N}{N}\right)$.}
Modifying $C$ appropriately yields the desired. \end{proof}
\begin{lem}\label{lem:parisi-1}
For {$k\geq 4$ even,}  $\lambda<\lambda_c$ and $\epsilon>0$, there are constants $C,c>0$ such that 
for every $N\geq1$, and $c>u\geq N^{-1/2+\epsilon}$, we
have 
\[
\E F_{2,N}\left(u,\frac{1}{N}{;\lambda}\right)\leq\lambda^{2}-Cu^{2}.
\]
\end{lem}
\begin{proof}
Observe that $\mathcal{P}(u,1, \Lambda)$ is $C^{2}$ in $(u,\Lambda)$ and $(0,0)$ is a critical
point with Hessian 
\[
	\operatorname{Hess}({\mathcal{P}})(0,0) = 
\left(\begin{array}{cc}
0 & -1 \\
-1 & {1}
\end{array}\right).
\]

This has an eigenvector of the form $(1,x)$ for some $x>0$ with
strictly negative eigenvalue $-\mu<0$. It follows that for $(u,\Lambda)=(u,ux)$ we have
\[
\mathcal P (u,1,ux)\leq \cP (0,1, 0)-K u^{2}=\lambda^{2}-K u^{2}
\]
 for $u\leq c$ for some $K,c>0$ independent of $N$. Combining this 
 with \eqref{eq:bound-1}, we obtain
 \[
 \E F_{2,N}(u,\eta{;\lambda}) \leq \lambda^2 - K u^2 +C\left(\eta+\frac{\log N}{N}\right).
 \]
 If we choose $\eta = 1/N$ and decrease $K$, the result follows since $u^2 \geq N^{-1+\eps}> \log N/N$ 
 for $\eps>0$ and $N\geq 1$.
\end{proof}
\begin{lem}\label{lem:parisi-2}
{For $\lambda<\lambda_{c}$ and $\epsilon>0$, there are $K,C>0$} such that for every
$u>\epsilon$, and  {$N\geq 1$} 
\[
\E F_{2,N}( u ,\frac{1}{N})\leq\lambda^2-K  + C\cdot\frac{\log N}{N}.
\]
\end{lem}
\begin{proof}
	{Notice that }
\begin{align*}
	\frac{d}{dm}\bigg|_{m=1} \cP (s ,m,0) &= -(s \xi'(s ) - \xi (s ))+\xi'(s )-\log(1+\xi'(s ))
		 {= \phi_{\lambda}(s),}
\end{align*}
where $\phi_{\lambda}$ is defined by \eqref{eq:def_phi}.
By Lemma~\ref{lem:gauss_scalar} we have $\phi_{\lambda}(s) < 0$ for all $\lambda < \lambda_c$ and $s \in (0,1]$.

Note that $\cP(u,1,0)=\lambda^2$. Thus for every $0<u \leq 1$,  $\cP (u,m ,{0})< \lambda^2$ for some $m> 1$.
{Observe that $\Phi(u) = \inf_m \cP(u,m,0)$ is upper-semicontinuous. Thus for any 
	$\eps>0$, there exists $K(\eps){>0}$ such that for all $u\in [\eps,1]$,
\[
\Phi(u) <\lambda^2 -K(\eps).
\]
In particular, for such $u$,} it follows that
\[
\E F_{2,N}(u,{\frac{1}{N}})\leq \lambda^2 - K(\epsilon)  + C\left( \frac{\log N}{N}\right)
\]
for $u > \epsilon$, which implies the desired result.
\end{proof}
Combining these two results, we obtain the following.
\begin{cor}\label{cor:parisi}
For $\lambda<\lambda_c$ and  $\eps>0$ sufficiently small, there is a $c>0$ such that
for $N\geq 1$,
\[
\E F_{2,N}(u,\frac{1}{N};\lambda)\leq \lambda^2 - c u^2,
\]
{for all $N^{-1/2+\eps}<u\leq 1$.}
\end{cor}
\begin{proof}
{Fix $\lambda$ and $\eps>0$.} By \prettyref{lem:parisi-1}, there is some $c_1,c_2>0$ such that for {all $N\geq 1$ and } 
$N^{-1/2+\epsilon}< u<c_1$, 
\[
\E F_{2,N}(u,\eta)\leq \lambda^2 - c_2 u^2
\]
Now for $c_1<u<1$, let $K(c_1)$ be as in \prettyref{lem:parisi-2}.
Then $K(c_1)u^2<K(c_1)$, so that, if we take $c = \min\{c_2, K(c_1)\}$
the result follows.
\end{proof}

\section{Estimation}
In this section, we prove Theorem~\ref{thm:max_likelihood}. We begin by providing a lower bound for the maximum likelihood for every $\lambda\geq 0$ using results on the ground state of the mixed $p$-spin model recently proved in \cite{jagannath2017low,chen2017parisi}.
We then use
the information-theoretic bound on the maximal correlation achievable by any
estimator from \cite{lesieur2017statistical} to obtain the matching upper bound.
We end by proving the desired result for the correlation $(\what{x}^{\rm ML}_{\lambda},X)$. In the remainder of this paper,
for ease of notation, we let 
\begin{equation}\label{eq:def_Hl}
	H_{\lambda}(x) = H(x) + \lambda N (x,X)^k,
\end{equation}
where $H(x)$ is as in \eqref{eq:ham-def}.

\subsection{Variational formula for the ground state of the mixed $p$-spin model}
We begin by recalling the following variational formula for the ground state of the mixed $p$-spin model. 
Consider the Gaussian process indexed by $x \in \mathbb{S}^{N-1}$:
$$
Y_N(x) = {\sqrt{N}} \sum_{p \geq 1} a_p \sum_{1 \leq i_1, \dots, i_p \leq N}  g_{i_1, \dots, i_p} x_{i_1} \dots x_{i_p},
$$
where $g_{i_1,\ldots,i_p}$ are i.i.d.\ standard Gaussian random variables and $\sum_{p \geq 1} 2^p a_p^2 < \infty$. 
The covariance of $Y_N$ is given by $$\E \big[Y_N(x) Y_N(y)\big] = {N}\xi((x,y)),$$ where
$\xi(t) = \sum_{p \geq 1} a_p^{{2}} t^p.$
Let $\cC$ denote the subset of $C([0,1])$ of functions that are positive, non-increasing and concave.
For any $h \geq 0$, we let $P_{h}:\cC\to \R$ be 
\[
	P_h(\phi) = \int \xi''(x)\phi(x) + \frac{1}{\phi(x)}dx + (h^2+\xi'(0)) \phi(0).
\]
Set 
\begin{equation}
\mathcal{G}(\xi,h)=\frac{1}{2} \min_{\phi\in \cC} P_h(\phi).
\end{equation}
Let us recall the following variational formula. For $x \in \mathbb S^{N-1}$, we write $x = (x_1, \dots , x_N)$.
\begin{thm}[\cite{chen2017parisi,jagannath2017low}]\label{thm:GS}
	For all $h \geq 0$,
	$$
	{\lim_{N\to\infty}\frac{1}{N}\max_{x \in \mathbb{S}^{N-1}} \Big\{ Y_N(x) + h \sqrt N \sum_{i=1}^N x_i \Big\}}
	= \mathcal{G}(\xi,h),
	$$
almost surely and in $L^1$.  
\end{thm}
\begin{rem}
While the results of \cite{chen2017parisi,jagannath2017low}
are stated with $\xi'(0)=0$, they still hold when $\xi'(0)>0$ by replacing $\xi\mapsto\xi(t)-\xi'(0) t$ and   $h^2\mapsto h^2+\xi'(0)$ .
To see this, simply note that the Crisanti-Sommers formula still holds in this setting by the main result of \cite{chen2013aizenman}. 
The reformulation from \cite[Eq. (1.0.1)]{jagannath2017low} is then changed by this replacement by simply 
repeating the integration by parts argument from \cite[Lemma 6.1.1]{jagannath2017low}. From here the arguments are unchanged under the above replacement.
\end{rem}

\subsection{The lower bound}

By Borell's inequality, the constrained maximum likelihood \eqref{eq:def_E} concentrates around its mean with sub-Gaussian tails. In particular, combining this with Borell-Cantelli we see that
\begin{equation}\label{eq:def_E2}
	E_{\lambda}(m) = \lim_{N \to \infty} \frac{1}{N} \E \Big[ \max_{x \in \mathbb{S}^{N-1}, \, (x,X) = m} \Big\{ \lambda N (x,X)^k + H(x) \Big\} \Big].
\end{equation}
Clearly,  $\liminf \frac{1}{N} \E \big[\max_{\mathbb{S}^{N-1}} H_{\lambda}\big] \geq E_{\lambda}(m)$ for all $m \in [-1,1]$. 
{Recall the definition of $\lambda_s$ from \eqref{eq:lambda_s-def} and $q_s(\lambda)$, see, e.g., Lemma~\ref{lem:gauss_scalar}.}
If we apply this for $\lambda > \lambda_s$ and $m  = \sqrt{q_s(\lambda)} > \sqrt{1-\frac{1}{k-1}}$ (by Lemma~\ref{lem:gauss_scalar}), Lemma~\ref{lem:up_qs} below will immediately yield the following lower bound. 
\begin{lem}\label{lem:lower_bound}
	For all $\lambda > \lambda_s$,
	\begin{equation}\label{eq:lower_bound}
	\liminf_{N \to \infty} \E\Big[\frac{1}{N}\max_{x \in \mathbb{S}^{N-1}} H_{\lambda}(x) \Big] \geq \sqrt{k}\frac{1 + \lambda^2 q_s(\lambda)^{k-1}}{\sqrt{1 + \lambda^2 k q_s(\lambda)^{k-1}}}.
	\end{equation}
\end{lem}

We now turn to the proof of \prettyref{lem:up_qs}. We begin by observing the following explicit representation for $E_\lambda$.

\begin{lem}
	For all $m \in [-1,1]$ the limit in \eqref{eq:def_E} exists and
\begin{equation}\label{eq:E-rs}
	E_{\lambda}(m)
		={\lambda m^k }+ \cG(\xi_m,0),
\end{equation} 
where $\xi_m(t) = (m^2 + (1-m^2)t )^k - m^{2k}$.
\end{lem}

\begin{proof}
We begin by observing that by rotational invariance,
$$
\frac{1}{N} \E \Big[ \max_{x \in \mathbb{S}^{N-1}, \, (x,X) = m} \Big\{ \lambda N (x,X)^k + H(x) \Big\}\Big]
= \lambda m^k +  \frac{1}{N} \E \Big[  \max_{x \in \mathbb{S}^{N-1}, \, x_1 = m} H(x) \Big].
$$

Let $x \in \mathbb{S}^{N-1}$ such that $x_1 = m$. Then 
\begin{align*}
	H(x)  \eqdist \sqrt N m^k g_{1, \dots, 1} + \sqrt N \sum_{j=0}^{k-1} \binom{k}{j}^{1/2} m^j \sum_{2 \leq i_1, \dots, i_{k-j} \leq N} g_{i_1, \dots, i_{k-j}} x_{i_1} \dots x_{i_{k-j}},
\end{align*}
where $\big( (g_{i_1, \dots, i_p})_{1 \leq i_1, \dots, i_p \leq N} \big)_{p \leq k}$ are  i.i.d.\\ standard Gaussians.

So that $\E \big[ \max_{x \in \mathbb{S}^{N-1}, \, x_1 = m} H(x) \big] = \E \big[ \max_{x \in \mathbb{S}^{N-2}} H_m(x)\big]$,
where $H_m$ is given by:
$$
H_m(x)
=\sqrt N \sum_{j=0}^{k-1} \binom{k}{j}^{1/2} m^j (1-m^2)^{(k-j)/2}\sum_{1 \leq i_1, \dots, i_{k-j} \leq N-1} g_{i_1, \dots, i_{k-j}} x_{i_1} \dots x_{i_{k-j}}.
$$
The function $H_m$ is a Gaussian process with covariance
\[
	\E\big[ H_m(x)H_m(y)\big] = N \xi_m( (x,y)),
\]
where  $\xi_m$  is given by
\begin{equation}\label{eq:xi-m}
\xi_m(t) 
= \sum_{j=0}^{k-1}\binom{k}{j} m^{2j} (1-m^2)^{k-j} t^{k-j}
=
(m^2 + (1-m^2)t )^k - m^{2k}.
\end{equation}
We conclude using Theorem~\ref{thm:GS} to obtain the result.
\end{proof}

We now observe that for {$m$} large enough, this formula has a particularly simple form. 

\begin{lem}
	For all $|m| \geq \sqrt{1 - \frac{1}{k-1}}$ we have:
	\begin{equation}\label{eq:E}
		E_{\lambda}(m) = \lambda m^k + \sqrt{k(1-m^2)}.
	\end{equation}
\end{lem}

\begin{proof}
	In the setting of Theorem~\ref{thm:GS} it was also shown in \cite{jagannath2017low,chen2017parisi} that if $\xi'(1) + h^2 \geq \xi''(1)$ then $\mathcal{G}(\xi,h) = \sqrt{\xi'(1) + h^2}$.
Since
\begin{align*}
	\xi_m'(t) &= k (1-m^2)(m^2 + (1-m^2)t )^{k-1}
	\\
	\xi_m''(t) &= k(k-1) (1-m^2)^2(m^2 + (1-m^2)t )^{k-2},
\end{align*}
the condition  $\xi_m'(1) \geq \xi_m''(1)$ corresponds to $(k-1) (1-m^2) \leq 1$, i.e.\ $|m| \geq \sqrt{1 - \frac{1}{k-1}}$. When this holds, we get that
$$
E_{\lambda}(m) = \lambda m^k + \cG(\xi_m,0)=\sqrt{\xi_m'(1)} = \lambda m^k + \sqrt{k (1-m^2)}
$$
by \eqref{eq:E-rs}.
\end{proof}

We end with the desired explicit formula for $E_\lambda(\sqrt q_s(\lambda))$. 
\begin{lem}\label{lem:up_qs}
	For all $\lambda > \lambda_s$, $\sqrt{q_s(\lambda)}$ is a local maximizer of $E_{\lambda}$ and if we write $x(\lambda) = \lambda^2 k q_s^{k-1}(\lambda)$,
	$$
	E_{\lambda}\big(\sqrt{q_s(\lambda)}\big) =  \frac{\sqrt{k}}{\sqrt{1 + x(\lambda)}} \Big(1 + \frac{x(\lambda)}{k}\Big).
	$$
\end{lem}
\begin{proof}
	Differentiating the expression \eqref{eq:E} for $m \geq \sqrt{1 - \frac{1}{k-1}}$ yields
	\begin{align*}
		E_{\lambda}'(m) 
&= \lambda k m^{k-1} - \frac{\sqrt{k}m}{\sqrt{1-m^2}}
= 
k \frac{1}{1-m^2}
\Big(\lambda k m^{k-1} + \frac{\sqrt{k}m}{\sqrt{1-m^2}} \Big)^{-1}
\Big(\lambda^2 k m^{2k-2} - \lambda^2 k m^{2k} - m^2 \Big)
	\end{align*}
	so that the functions $\phi_{\lambda}$, $f_{\lambda}$ and $m^2 \mapsto E_{\lambda}(m)$ have precisely the same monoticity on $[1 - \frac{1}{k-1},1)$ (recall the expression of the derivatives $f_{\lambda}'$ and $\phi_{\lambda}'$ given by \eqref{eq:der_phi0}). Lemma~\ref{lem:gauss_scalar} gives that $q_s(\lambda)$ is a local maximum of $f_{\lambda}$ and $\phi_{\lambda}$ for $\lambda > \lambda_s$, $\sqrt{q_s(\lambda)}$ is therefore a local maximum of $E_{\lambda}$.

	Let us now compute $E_{\lambda}(\sqrt{q_s(\lambda)})$.
	By Lemma~\ref{lem:gauss_scalar}, $q_s(\lambda) = \frac{x(\lambda)}{1+x(\lambda)}$. Consequently,
	\begin{align*}
		E_{\lambda}(q_s(\lambda)^{1/2}) 
	&= \lambda q_s(\lambda)^{k/2} + \sqrt{k(1 - q_s(\lambda))}
	= \frac{\sqrt{k}}{\sqrt{1 + x(\lambda)}} \Big(1 + \frac{x(\lambda)}{k}\Big).\qedhere
	\end{align*}
\end{proof}

\subsection{The upper bound}

We prove here the upper bound. 
\begin{lem}\label{lem:upper_bound}
	For all $\lambda \geq 0$,
	\begin{equation}\label{eq:upper_bound}
	\limsup_{N \to \infty} \E\Big[\frac{1}{N}\max_{x \in \mathbb{S}^{N-1}} H_{\lambda}(x) \Big] \leq {\rm GS}_k + \int_0^\lambda q_*(t)^{k/2}\, dt.
 	\end{equation}
\end{lem}
We defer the proof of this momentarily to observe the following information-theoretic bounds which will be useful
in its proof. 
\begin{prop}\label{prop:IT}
	Assume that $X$ is uniformly distributed over $\bS^{N-1}$, independently from $W$. Then for all $\lambda \in (0, +\infty) \setminus \{\lambda_c \}$
	$$
	\lim_{N \to \infty}
	\E \Big[
		\Big\| X^{\otimes k} - \E\big[ X^{\otimes k} \big| Y \big] \Big\|^2
	\Big]
	= 
1 - q_*(\lambda)^k.
	$$
\end{prop}
This result follows from \cite{lesieur2017statistical,barbier2017stochastic} by approximating the uniform measure on $\mathbb S^N$ by an i.i.d.\ Gaussian measure.  
For the completeness, we provide a proof in \prettyref{app:proof-prop-IT}. As a consequence of this, we have the following.
\begin{cor}\label{cor:upper_IT}
	Assume that $X$ is uniformly distributed over $\bS^{N-1}$, independently from $W$.
	Then for all measurable functions $\what{x}: (\R^N)^{\otimes k} \to \mathbb{S}^{N-1}$ and for all $\lambda \neq \lambda_c$ we have
	$$
	\limsup_{N \to \infty} \E \Big[ \big( \what{x}(Y), X \big)^k  \Big] \leq q_*(\lambda)^{k/2}.
	$$
\end{cor}
\begin{proof}
	Compute
	\begin{align*}
		 \E \Big[\Big\| X^{\otimes k} - \big(\sqrt{q_*(\lambda)} \what{x}(Y)\big)^{\otimes k} \Big\|^2 \Big]
		&=
		\E \big[\big\| X^{\otimes k}\big\|^2 \big]
		+
		q_*(\lambda)^{k} \E \big[\big\| \what{x}(Y)^{\otimes k}\big\|^2 \big]
		- 2 q_*(\lambda)^{k/2} \E \Big[ \big( \what{x}(Y), X \big)^k \Big]
		\\
		&= 1 + q_*(\lambda)^{k}
		- 2 q_*(\lambda)^{k/2}\E \Big[ \big( \what{x}(Y), X \big)^k \Big]. 
	\end{align*}
	{Recall that the posterior mean, $\E( X^{\tensor k}\vert Y)$, uniquely 
	achieves the minimal mean squared error over all square-integrable tensor-valued estimators, $\what{T}(Y)$,  for $X^{\tensor k}$.}
	The proposition follows then from Proposition~\ref{prop:IT} which gives
	$$
	\liminf_{N \to \infty}
	\E \Big[\Big\| X^{\otimes k} - \big(\sqrt{q_*(\lambda)} \what{x}(Y)\big)^{\otimes k} \Big\|^2 \Big]
	\geq \liminf_{N \to \infty}
	\E \Big[
		\Big\| X^{\otimes k} - \E\big[ X^{\otimes k} \big| Y \big] \Big\|^2
	\Big]
	= 1 - q_*(\lambda)^k.\qedhere
	$$
\end{proof}

With this in hand we may now prove \prettyref{lem:upper_bound}.

\begin{proof}[\bf Proof  of \prettyref{lem:upper_bound}]
	By Proposition~\ref{prop:unique_ML} and an application of an envelope-type theorem (see, e.g., Proposition~\ref{prop:envelope_compact}),  {the map} $\lambda \mapsto \frac{1}{N}\E \big[\max_{\mathbb{S}^{N-1}} H_{\lambda}(x) \big]$ {is differentiable for $\lambda\geq 0$}, with derivative
	\begin{equation}\label{eq:M_N-diff}
	\frac{\partial}{\partial \lambda} \E\Big[\frac{1}{N} \max_{x \in \mathbb{S}^{N-1}} H_{\lambda}(x)\Big]
	=\E \Big[
		{(\what{x}_{\lambda}^{\rm ML}, X )^k
	\Big].}
	\end{equation}
By \cite{jagannath2017low,chen2017parisi} we know that $\frac{1}{N} \E [\max_{\mathbb{S}^{N-1}} H_0\big] \rightarrow {\rm GS}_k$.
	The reverse Fatou lemma gives then
	\begin{align*}
		\limsup_{N \to \infty} \E\Big[\frac{1}{N} \max_{x \in \mathbb{S}^{N-1}} H_{\lambda}(x)\Big]
		\leq  \int_0^{\lambda} \limsup_{N \to \infty}\E \Big[
	\big( \what{x}_{\gamma}^{\rm ML}, X \big)^k
\Big]\, d\gamma + {\rm GS}_k \leq \int_0^{\lambda} q_*(\gamma)^{k/2}\, d\gamma + {\rm GS}_k,
\end{align*}
where the second inequality follows from \prettyref{cor:upper_IT}.
\end{proof}

\subsection{Proof of first part of Theorem~\ref{thm:max_likelihood} }\label{sec:proof_th_max_likelihood}

By an elementary but tedious calculation (see \mbox{\prettyref{lem:key_identities}}) the right sides of \eqref{eq:lower_bound} and \eqref{eq:upper_bound} are equal for $\lambda\geq \lambda_c$ 
(recall that $q_*(\lambda)=q_s(\lambda)$ for such $\lambda$).
Thus for all $\lambda > \lambda_c$,
\begin{equation}\label{eq:cv_up}
\E\Big[\frac{1}{N}\max_{x \in \mathbb{S}^{N-1}} H_{\lambda}(x) \Big] \xrightarrow[N \to \infty]{} \sqrt{k}\frac{1 + \lambda^2 q_*(\lambda)^{k-1}}{\sqrt{1 + \lambda^2 k q_*(\lambda)^{k-1}}}.
\end{equation}

We will now prove that for $\lambda \leq \lambda_c$, $\widebar{M}_N(\lambda) \xrightarrow[N \to \infty]{} {\rm GS}_k$, where $\widebar{M}_N(\lambda)$ is defined by
\[
\widebar{M}_N (\lambda) =  \E \left[\frac{1}{N} \max_{x \in \mathbb{S}^{N-1}} H_{\lambda}(x)\right].
\]
Notice that $\widebar{M}_N(\lambda)$ is convex as an expectation of a maximum of linear functions. By \eqref{eq:M_N-diff}, it follows that
 $\widebar{M}_N'(0^+)\geq 0$. (When $k$ is odd, we use rotational invariance to see that it is in fact zero.)
Consequently, $\widebar{M}_N$ is non-decreasing on $[0,+\infty)$. 

By \cite{jagannath2017low,chen2017parisi} (see  \prettyref{thm:GS}), $\lim_{N \to \infty} \widebar{M}_N(0) = {\rm GS}_k$.  By \eqref{eq:cv_up} and Lemma~\ref{lem:key_identities}, \begin{equation*}\lim_{\lambda \to \lambda_c^+} \lim_{N \to \infty} \widebar{M}_N(\lambda) = {\rm GS}_k.\end{equation*} Consequently, we obtain that for all $\lambda \in [0,\lambda_c]$, $\lim_{N \to \infty} \widebar{M}_N(\lambda) = {\rm GS}_k$.

The almost sure convergence of \eqref{eq:lim_max_likelihood} follows then from the convergence of the expectation $\widebar{M}_N(\lambda)$, combined with Borell's inequality for suprema of Gaussian processes (see for instance  \cite[Theorem 5.8]{boucheron2013concentration}) and the Borel-Cantelli Lemma.\qed

\begin{rem}\label{rem:rigorous}
By \eqref{eq:E-rs} , $E_\lambda$ is given by a variational problem over the space $\cC$. 
We first observe that one can easily solve this variational problem numerically due to the following simple reductions.
First note that if we let $\xi_m$ be as in \eqref{eq:xi-m}, then $(1/\sqrt{\xi_m''})''$  is strictly positive,  where the prime denotes differentiation in $t$. Thus by \cite[Theorem 1.2.4]{jagannath2017low}, the minimizer $\phi$ must be of the form $\phi(s)=\int_s^1 d\nu_s$ where $\nu_s = \theta_1 \delta_a + \theta_2 \delta_b$, where $a,b\in [0,1]$ and $\theta_i\geq 0$. 
Thus the variational problem \eqref{eq:E-rs} is a variational problem over 4 parameters which can be solved numerically. 
These observations then rigorously justify the starting point of the discussion in \cite[Section 4]{BBCR18}, namely the ``RS" and ``1RSB" calculation in \cite[Sect. 4.B]{BBCR18} in the regime they analyize, called the ``$T=0$" regime there. We refer the reader there for a more in-depth discussion, see \cite[Sect. 4.C]{BBCR18}. 
\end{rem}

\subsection{Proof of second part of Theorem~\ref{thm:max_likelihood}}\label{sec:proof_cor_max_likelihood}
We now turn to the second part of \prettyref{thm:max_likelihood}, namely \eqref{eq:correlation}.
Let $M_N(\lambda)$ denote
\[
M_N(\lambda) = \frac{1}{N} \max_{x \in \mathbb{S}^{N-1}} H_{\lambda}(x).
\]
Fix $\lambda > 0$.
By \eqref{eq:lim_max_likelihood} and Lemma~\ref{lem:key_identities}
\[
\lim_{N\to\infty}M_N(\lambda)  = \ell(\lambda) = 
\begin{cases}
	{\rm GS}_k &  \text{if} \quad \lambda \leq \lambda_c \\
	{\rm GS}_k + \int_{0}^{\lambda} q_*(\gamma)^{k/2} \, d\gamma & \text{if} \quad \lambda > \lambda_c.
\end{cases}
\]
Let $\lambda \in (0, +\infty) \setminus \{ \lambda_c \}$. 
By  Proposition~\ref{prop:unique_ML} and the Milgrom-Segal envelope theorem (see \prettyref{prop:envelope_compact}), $M_N$ is  differentiable in $\lambda$ with derivative
$$
M_N'(\lambda) = \big(\what{x}^{\rm ML}_{\lambda}, X  \big)^k,
$$
almost surely. 
As $M_N$ is convex in $\lambda$ (it is a maximum of linear functions), we see that
for any $0<h<\lambda$,
$$
\frac{M_N(\lambda - h) - M_N(\lambda)}{h} \leq M_N'(\lambda) \leq \frac{M_N(\lambda + h) - M_N(\lambda)}{h}.
$$
By taking the $N \to \infty$ limit, we get that almost surely
\begin{equation}\label{eq:limsup_as}
	\frac{\ell(\lambda - h) - \ell(\lambda)}{h} \leq 
	\liminf_{N \to \infty} ( \what{x}^{\rm ML}_{\lambda}, X )^k
	\leq
	\limsup_{N \to \infty} ( \what{x}^{\rm ML}_{\lambda}, X )^k
	\leq \frac{\ell(\lambda + h) - \ell(\lambda)}{h}.
\end{equation}
Since $\ell$ is differentiable for $\lambda \neq \lambda_c$, we may take $h\to0$ to obtain
$
\lim_{N \to \infty}
( \what{x}^{\rm ML}_{\lambda}, X )^k = q_*(\lambda)^{k/2}
$
almost surely, which proves \eqref{eq:correlation}.\qed

\begin{appendices}
	\section{Appendix}\label{sec:appendix}
	\subsection{Uniqueness of minimizers and envelope theorems}\label{app:elementary}

	This section gathers some basic lemmas that will be useful for the analysis.
	\begin{prop}\label{prop:unique_ML} Recall the definition \eqref{eq:def_Hl} of $H_{\lambda}$. We have the following
		\begin{itemize}
			\item If $k$ is odd, then $H_{\lambda}$ has almost surely one unique maximizer over $\mathbb{S}^{N-1}$.
			\item If $k$ is even, then $H_{\lambda}$ has almost surely two maximizers over $\mathbb{S}^{N-1}$, $x^*$ and $-x^*$.
		\end{itemize}
	\end{prop}
	\begin{proof}
We note the following basic fact from the theory of Gaussian processes,
see, e.g.\  \cite{kim1990cube}.
		\begin{lem}\label{lem:unique-min-gaussian}
			Let $(Z(t))_{t \in T}$ be a Gaussian process indexed by a compact metric space $T$ such that $t \mapsto Z(t)$ is continuous almost surely.
			If the intrinsic quasi-metric, $d(s,t)^2 = \Var\big(Z(s) - Z(t)\big)$, is a metric, i.e., $d(s,t) \neq 0$ for $s\neq t$, then $Z$ admits a unique maximizer on $T$ almost surely.
		\end{lem}
		Observe $H_{\lambda}$ is continuous on the compact $\mathbb{S}^{N-1}$. For $x_1,x_2 \in \mathbb{S}^{N-1}$, we have
		$$
		\Var\big(H_{\lambda}(x^1) - H_{\lambda}(x^2)\big)
		= 2 N \big(1 - ( x^1, x^2 )^k  \big).
		$$
		If $k$ is odd, then the proposition follows directly from the Lemma. If $k$ is even, we apply the Lemma on the quotient space $\mathbb{S}^{N-1} / \sim$ where $\sim$ denotes the equivalence relation defined by $x^1 \sim x^2 \Leftrightarrow \big(x^1 = x^2 \ \text{or} \ x^1 = - x^2\big)$.
	\end{proof}

	We recall the following envelope theorem of Milgrom and Segal~\cite{milgrom2002envelope}.
	Let $X$ be a set of parameters and consider a function $f: X \times [0,1] \to \R$. Define, for $t \in [0,1]$
	\begin{align*}
		V(t) &= \sup_{x \in X} f(x,t) \,,\\
		X^*(t) &= \big\{ x \in X \, \big| \, f(x,t) = V(t) \big\} \,.
	\end{align*}

	\begin{prop}[Corollary~4 from~\cite{milgrom2002envelope}\,] \label{prop:envelope_compact}
		Suppose that $X$ is nonempty and compact. Suppose that for all $t\in [0,1]$, $f(\cdot,t)$ is continuous. Suppose also that $f$ admits a partial derivative $f_t$ with respect to $t$ that is continuous in $(x,t)$ over $X \times [0,1]$. Then
		\begin{itemize}
			\item $\displaystyle V'(t^+) = \max_{x^* \in X^*(t)} f_t(x^*,t)$ for all $t\in [0,1)$ and $\displaystyle V'(t^-) = \min_{x^* \in X^*(t)} f_t(x^*,t)$ for all $t\in (0,1]$.
			\item $V$ is differentiable at $t \in (0,1)$ is and only if \,$\displaystyle \Big\{ f_t(x^*,t) \, \Big| \, x^* \in X^*(t) \Big\}$
				is a singleton. In that case $V'(t) = f_t(x^*,t)$ for all $x^* \in X^*(t)$.
		\end{itemize}
	\end{prop}

	\subsection{Study of the asymptotic equations}\label{app:asymptotic}
Define, for all $q \in [0,1]$
\begin{equation}\label{eq:def_phi}
	\phi_{\lambda}(q) = \lambda^2 k q^{k-1} - \log(1 + \lambda^2 k q^{k-1}) - \lambda^2 (k-1) q^k.
\end{equation}
\begin{lem}\label{lem:gauss_scalar}
	We have for all $\lambda > 0$,
	$$
	\max_{q \in [0,1)} f_{\lambda}(q) = \max_{q \in [0,1]} \phi_{\lambda}(q)
	$$
	Furthermore, if we let $\lambda_s = \sqrt{\frac{(k-1)^{k-1}}{k (k-2)^{k-2}}}$:
	\begin{itemize}
		\item For $\lambda < \lambda_s$, then the functions $f_{\lambda}$ and $\phi_{\lambda}$ are decreasing on $[0,1)$.
					\item For $\lambda > \lambda_s$, the functions $f_{\lambda}$ and $\phi_{\lambda}$ have a strict local minimum at $q_u(\lambda)$
						and a strict local maximum at $q_s(\lambda)$ where $0<q_u< \frac{k-2}{k-1}<q_s<1$, {and both functions are} strictly monotone on the intervals $(0,q_u)$, $(q_u,q_s)$ and $(q_s,1)$. Moreover, $q_s(\lambda)$ {is strictly increasing in $\lambda$} and satisfies:	
			\begin{equation}\label{eq:q_SE}
				q_s(\lambda) = \frac{\lambda^2 k q_s(\lambda)^{k-1}}{1 + \lambda^2 k q_s(\lambda)^{k-1}} \,.
			\end{equation}
	\end{itemize}
	Finally, for $\lambda > \lambda_c$,  $q_*(\lambda) = q_s(\lambda)$ is the unique maximizer of $f_{\lambda}$ and $\phi_{\lambda}$ over $[0,1)$.
\end{lem}
\begin{proof}
	We have for $q \in [0,1)$
	\begin{align}\label{eq:der_phi0}
	\phi_{\lambda}'(q) = 
	\frac{%
	k(k-1) \lambda^2 q^{k-1}}{%
1 + \lambda^2 k q^{k-1}}
h(q)
\qquad \text{and} \qquad
	f_{\lambda}'(q) = \frac{h(q)}{1-q}
	\end{align}
	where $h(q) = \lambda^2 k q^{k-1} - \lambda^2 k q^{k} - q$.
	It suffices therefore to study the variations of $f_{\lambda}$. Notice also that
	$$
	\phi_{\lambda}(q) = f_{\lambda}(q) + h(q) - \log(1 + h(q)).
	$$
	{Since $f_\lambda'(q) = 0$ implies $h(q)=0$,} this implies that 
	$$
	\max_{q \in [0,1]} \phi_{\lambda}(q) = \max_{q \in [0,1)} f_{\lambda}(q)
	$$
	and that these maxima are achieved at the same points. 
	Let us now study the sign of the polynomial $h(q)$:
	\begin{equation}\label{eq:si}
		h(q) = qk \lambda^2 \big(q^{k-2} - q^{k-1} - \frac{1}{k \lambda^2}\big).
	\end{equation}
	One verifies easily that the polynomial $q^{k-2} - q^{k-1}$ achives its maximum at $\frac{k-2}{k-1}$ and that the value of this maximum is $\frac{(k-2)^{k-2}}{(k-1)^{k-1}}$.
	We get that for $\lambda < \lambda_s$, $h'(q) < 0$ for all $q>0$. For $\lambda > \lambda_s$ we get that $h$ admits exactly 3 zeros on $\R$: $0< q_u( \lambda ) < q_s(\lambda)<1$. 
	Since the maximum of $h$ is achieved at $\frac{k-2}{k-1}$ we get that $q_u(\lambda) < \frac{k-2}{k-1} < q_s(\lambda)$.
	{
		$q_u<q_s$ are the positive roots of $q^{k-2} - q^{k-1} = \frac{1}{k \lambda^2}$: $q_u$ is therefore strictly decreasing and $q_s$ is strictly increasing in $\lambda$.
	}
	This proves the two points of the lemma; \eqref{eq:q_SE} simply follows from the fact that $h(q_s(\lambda)) = 0$.
	The last statement of Lemma~\ref{lem:gauss_scalar} is then an immediate consequence of the definition of $\lambda_c$.
\end{proof}

Recall that $z_k$ is defined as the unique zero of $\varphi_k(z) = \frac{1+z}{z^2} \log(1+z) - \frac{1}{z} - \frac{1}{k}$ on $(0,+\infty)$.

\begin{lem}\label{lem:x_k}
	The mapping $\lambda \mapsto q_s(\lambda)$ is $\cC^{\infty}$ on $(\lambda_s, + \infty)$.
	Moreover
	$\lambda^2 k q_s(\lambda_c)^{k-1} = z_k$.
\end{lem}
\begin{proof}
	The first part follows from a straightforward application of the implicit function theorem.

	We get in particular that the mapping $\lambda \mapsto q_s(\lambda)$ is continuous for $\lambda > \lambda_s $. So by definition of $\lambda_c$ and Lemma~\ref{lem:gauss_scalar}, $\phi_{\lambda_c}(q_s(\lambda_c)) = 0$.
	Let us write $x = \lambda^2 k q_s(\lambda_c)^{k-1}$. 
	\begin{align*}
		0 = \phi_{\lambda_c}(q_s(\lambda_c)) =
		x - \log(1+ x) - \frac{k-1}{k} x q_s(\lambda_c)
		=
		x - \log(1+ x) - \frac{k-1}{k} \frac{x^2}{1+x} \,,
	\end{align*}
	because $q_s(\lambda_c) = \frac{x}{1+x}$ (see \eqref{eq:q_SE}).
	This gives that $\varphi_k(x) = 0$ and thus $x=z_k$.
	\end{proof}

\begin{lem}\label{lem:key_identities}
	Let $\lambda > \lambda_c$ and write $x(\lambda) = \lambda^2 k q_s(\lambda)^{k-1}$. Then we have
	$$
	\frac{\sqrt{k}}{\sqrt{1+x(\lambda)}} \Big(1 + \frac{x(\lambda)}{k}\Big)
	=	{\rm GS}_k + \int_{\lambda_c}^{\lambda} q_s(\gamma)^{k/2} d\gamma
	\,.
	$$
\end{lem}
\begin{proof}
	Let us write $g(\lambda) = \frac{\sqrt{k}}{\sqrt{1+x(\lambda)}} \big(1 + \frac{x(\lambda)}{k}\big)$.
	By Lemma~\ref{lem:up_qs}, $\sqrt{q_s(\lambda)}$ is a local maximizer of $E_{\lambda}$ and thus a critical point of $E_{\lambda}$. This gives
	$$
	g'(\lambda) = \partial_{\lambda} \Big[E_{\lambda}(\sqrt{q_s(\lambda)})\Big]
	= \partial_{\lambda} E_{\lambda}(\sqrt{q_s(\lambda)}) = q_s(\lambda)^{k/2}.
	$$
	The lemma follows then from the fact that $x(\lambda_c) = z_k$ by Lemma~\ref{lem:x_k} and the definition \eqref{eq:def_GS} of ${\rm GS}_k$.
\end{proof}

\subsection{Proof of \prettyref{prop:IT}}\label{app:proof-prop-IT} 
For $P_0$ a probability distribution over $(\R^N)^{\otimes k}$ with finite second moment, we define the free energy
$$
F_{P_0}(\gamma) =
\frac{1}{N} \E \log \int P_0(dx) \exp\Big( 
	\sqrt{\gamma N}( x , W )
	+ {\gamma}N ( x, X_0  )
	- \frac{1}{2} {\gamma} N \| x\|^{2}
\Big)
$$
where $X_0 \sim P_0$ and ${W}_{i_1,\dots,i_k} \sim \cN(0,1)$ are independent. Proposition~A.1 from~\cite{miolane2018phase} states that for two probability distributions $P_1$, $P_2$ on $(\R^N)^{\otimes k}$ with finite second moment, we have
$$
\big| F_{P_1}(\gamma) - F_{P_2}(\gamma) \big|
\leq \frac{\gamma}{2} \Big(\sqrt{\E_{P_1} \|X_1\|^2} + \sqrt{\E_{P_2} \| X_2 \|^2}\Big) W_2(P_1,P_2),
$$
where $W_2(P_1,P_2)$ denotes the Wasserstein distance of order $2$ between $P_1$ and $P_2$. 
Let $\mu_N$ be the distribution of $X^{\otimes k}$ when $X \sim {\rm Unif}(\mathbb{S}^{N-1})$ and let $\nu_N$ be the distribution of $X^{\otimes k}$ when $X \sim {\cN(0,\frac{1}{N}Id_N)}$.
Let us compute a bound on $W_2(\mu_N, \nu_N)$. Let {$X $ be drawn uniformly over $\mathbb S^{N-1}$}, and $G \sim {\cN(0,Id)}$, independently from $X$. Then {$(X^{\tensor k}, (\|G\| X/\sqrt{N})^{\tensor k} )$} is a coupling of $\mu_N$ and $\nu_N$, so that, by definition $W_2$,
\begin{align*}
W_2(\mu_N, \nu_N)^2 \leq
	\E \big\| X^{\otimes k} - \big( X \|G\| /\sqrt{N} \big)^{\otimes k} \big\|^2
	= \E \Big[\Big(\Big( \frac{1}{N} \sum_{i=1}^N G_i^2\Big)^{k/2} - 1 \Big)^2 \Big]
\end{align*}
{where we use that $\E\norm{X}^k=1$. By the law of large numbers, it then follows that 
\[
\lim_{N\to\infty} \abs{F_{\mu_N}(\gamma)-F_{\nu_N}(\gamma)}\to0.
\]}
Recall the definition \eqref{eq:def_phi} of $\phi_{\lambda}(q)$ and define 
$L(\gamma) = \frac{1}{2}\max_{q \in [0,1]} \phi_{{\sqrt{\gamma}}}(q) = \frac{1}{2} \max_{q \in [0,1)} f_{{\sqrt{\gamma}}}(q)$, where the equality comes from Lemma~\ref{lem:gauss_scalar}.
Now,~\cite[Theorem 1]{lesieur2017statistical} gives that for all $\lambda \geq 0$, $F_{\nu_N}(\gamma) \rightarrow L(\gamma)$ as $N\rightarrow \infty$, which implies $F_{\mu_N}(\gamma) \rightarrow L(\gamma)$. 

The ``I-MMSE Theorem'' from~\cite{guo2005mutual}
(see \cite[Proposition~1.4]{miolane2018phase} for a statement {of this result} closer to the notations used here) gives that $\gamma \mapsto F_{\mu_N}(\gamma)$ is convex and differentiable over $[0, + \infty)$ and
$$
F_{\mu_N}'(\lambda^2) = \frac{1}{2} \left(
	1 -
	\E \Big[
		\big\| X^{\otimes k} - \E\big[ X^{\otimes k} \big| Y \big] \big\|^2
	\Big]
\right).
$$
{By Griffith's lemma for convex functions, $F_{\mu_N}'$ converges to {$L'$ } for each $\lambda > 0$ at which $L$ is differentiable.} For {$\gamma < \lambda_c^2$}, $L(\gamma) = 0$, so $L$ is differentiable on {$(0,\lambda_c^2)$} with derivative equal to $0$. For {$\gamma > \lambda_c^2$}, we know by Lemma~\ref{lem:gauss_scalar} that $f_{{\sqrt{\gamma}}}$ admits a unique maximizer $q_*({\sqrt{\gamma}})$ on $[0,1]$. Proposition~\ref{prop:envelope_compact} gives that $L$ is differentiable at $\gamma$ with derivative
\begin{align*}
	L'(\gamma) 
	&= \frac{1}{2} (\partial_{\gamma} f_{\sqrt{\gamma}}) (q_*(\sqrt \gamma))
	= \frac{1}{2} q_*(\sqrt \gamma)^k.
\end{align*}
We conclude that 
$$
\lim_{N \to \infty} \frac{1}{2} \left(
	1 - 
	\E \Big[
		\big\| X^{\otimes k} - \E\big[ X^{\otimes k} \big| Y \big] \big\|^2
	\Big]
\right)
= 
\lim_{N \to \infty}
F_{\mu_N}'(\lambda^2) =
\begin{cases}
	0 & \text{if} \ \ \lambda < \lambda_c \\
	\frac{1}{2} q_*(\lambda)^k & \text{if} \ \ \lambda > \lambda_c.
\end{cases}\qed
$$

\subsection{Elementary lemmas}
We collect here the following elementary lemmas which are used in the above.
\begin{lem}\label{lem:chi_lb}
Let $\{X_i\}$ be standard normal random variables, and let 
$S_N = \frac{1}{bN} \sum_{i=1}^N {X^2_i}.$
{There is a $C>0$ such that for $N\geq 1$ and $b> 1$ (possibly varying in $N$)}, 
\begin{equation}\label{eq:chi-square-lower-bound}
 \frac{1}{N } \log \prob ( S_N \ge 1 ) \ge -  \frac{1}{2}(b-1-\log b)  - \frac{2 \log b}{N} - C \cdot\frac{\log N}{N}  .
 \end{equation}
\end{lem}
\begin{proof}
If we let $K=N/2$, then
\[\prob ( S_N \ge 1 ) = \frac{1 }{\Gamma(K)}  \int_{bK}^\infty y^{K-1} e^{-y}\, dy.\]
In the integrand, we may bound $ y^{K-1} \ge  (bK)^{K-1}$, yielding
\[\prob ( S_N \ge 1 ) \ge \frac{(bK)^{K-1}}{\Gamma(K)}  \int_{bK}^\infty  e^{-y}\, dy = \frac{(bK)^{K-1}}{\Gamma(K)}  e^{-bK} .\]
By Stirling's approximation, it then follows that
\begin{align*} 
\frac{1}{K} \log \prob ( S_N \ge 1 ) 
									  &{\geq }  - (b-1-\log b)  - \frac{\log b}{K} {- } C\cdot\frac{\log K}{K}, 
\end{align*}
for some $C>0$ from which the result follows. 
\end{proof}
The following result is an elementary consequence of Gaussian integration by parts. For a proof
in the discrete setting, see, e.g., \cite[Lemma 1.1]{PanchSKBook}. The proof in our setting then follows by an elementary approximation argument.
\begin{lem}
Let $a(x)$ and $b(x)$ be centered Gaussian processes on $\mathbb S^{N-1}$ for any $N\geq 1$,
with smooth covariances, continuous mutual covariance 
\[
C(x^1,x^2) = \E a(x^1)b(x^2),
\]
which is assumed to be smooth and such that $\E \max a(x)$ is finite. Then if we let ${\pi(dx) = \exp(b(x)){dx}/Z}$, where $Z$ is 
chosen so that this is a probability measure,
\begin{equation}\label{eq:GGIBP}
\E \int a(x)\, d\pi = \E \int\int C(x^1,x^1)- C(x^1,x^2)\, d\pi^{\tensor 2}
\end{equation}
\end{lem}
\begin{proof}
By the assumption on the covariances, the processes $a(x)$ and $b(x)$ are a.s. smooth \cite{AdlerTaylor}.
By the law of large numbers, there is a collection of points $(y^\ell)\in\mathbb S^{N-1}$, such that the empirical measure
\[
\cE_n = \frac{1}{n}\sum_{\ell=1}^n \delta_{y^\ell}
\]
converges weak-* to the uniform measure. Evidently, if we let
\[
\pi_n = \frac{\sum_{\ell=1}^n \exp(b(y^\ell))\delta_{y^\ell}}{\sum_{\ell=1}^n \exp(b(y^\ell))},
\]
then $\pi_n\to\pi$ weak-* a.s. By Gaussian integration by parts, \cite[Lemma 1.1]{PanchSKBook},
\[
\E \int a(x) d\pi_n = \E \int \int C(x^1,x^1)- C(x^1,x^2)d\pi_n^{\tensor 2}.
\]
Since, $\max a(x)$, has bounded mean. The result then follows by applying
the dominated convergence theorem to each side of this equality.   
\end{proof}

\end{appendices}

\bibliography{TensorPCA,references.bib} 

\bibliographystyle{amsplain}

\end{document}